\documentclass[a4paper,11pt]{amsart}
\usepackage{amsthm,amssymb,mathtools,amsfonts,latexsym,rawfonts,amsmath, mathrsfs, lscape}
\usepackage{stmaryrd,tikz,pgfplots}
\usepackage{marginnote}  

\usepackage[top=1.5in, bottom=0.9in, left=0.9in, right=0.9in]{geometry}

\usepackage[textsize=tiny]{todonotes}

\newcommand{\Marginpar}[1]{\marginpar{\tiny{#1}}}
\newcommand{\Note}[1]{{\par\noindent\hrulefill\par\tiny{#1}\par\noindent\hrulefill\par}}
\newcommand{\Detail}[1]{{#1}}
\renewcommand{\Marginpar}[1]{}
\renewcommand{\Note}[1]{}
\renewcommand{\Detail}[1]{}

\usepackage{hyperref}
\usepackage[all]{xy}
\usepackage{amsmath,amsfonts}	
\usepackage{amsthm,latexsym,amssymb}
\usepackage{tensor}

\usepackage{array, tabularx}

\usepackage{booktabs} 
\usepackage{siunitx}

\usepackage{tikz-cd}

\newtheorem{thm}{Theorem}
\newtheorem*{thm*}{Theorem}
\newtheorem{prop}[thm]{Proposition}
\newtheorem{lem}[thm]{Lemma}
\newtheorem*{lem*}{Lemma}
\newtheorem{cor}[thm]{Corollary}

\theoremstyle{definition}
\newtheorem{defn}[thm]{Definition}
\newtheorem*{defn*}{Definition}

\newtheorem{rem}[thm]{Remark}

\renewcommand{\[}{\begin{equation*}}
\renewcommand{\]}{\end{equation*}}

\DeclareMathOperator\tr{tr}

\def \R {\mathbb R}

\def \g {\mathfrak g}

\def \n {\mathfrak n}
\def \Na {\nabla}
\def \tr {\mbox{tr}}
\def \ad {\mbox{ad}}

\begin{document}
\parskip1mm

\title[Second-Chern--Einstein metrics]{Second-Chern--Einstein metrics on $4$-dimensional almost-Hermitian manifolds}

\author{Giuseppe Barbaro}
\address{Dipartimento di Matematica ``Guido Castelnuovo'', Università La Sapienza di Roma, Piazzale Aldo Moro 5, 00185 Roma}
\email{g.barbaro@uniroma1.it}

\author{Mehdi Lejmi}
\address{Department of Mathematics, Bronx Community College of CUNY, Bronx, NY 10453, USA.}
\email{mehdi.lejmi@bcc.cuny.edu}

\thanks{The authors are thankful to Daniele Angella for his comments. The first author would also like to thank the CUNY Graduate Center of New York, Mehdi Lejmi and Scott Wilson for their warm hospitality. The first author is supported by GNSAGA of INdAM and by project PRIN2017 ``Real and Complex Manifolds: Topology, Geometry and holomorphic dynamics'' (code 2017JZ2SW5). The second author is supported by the Simons Foundation Grant \#636075.}

\keywords{Almost-Hermitian Metrics, Chern--Einstein metrics, Hermitian connections, Weyl connection. }

\subjclass[2010]{53C55 (primary); 53B35 (secondary)} 

\begin{abstract}
    We study $4$-dimensional second-Chern--Einstein almost-Hermitian manifolds. In the compact case, we observe that under a certain hypothesis the Riemannian dual of the Lee form is a Killing vector field. We use that observation to describe $4$-dimensional compact second-Chern--Einstein locally conformally symplectic manifolds and we give some examples of such manifolds. Finally, we study the second-Chern--Einstein problem on unimodular almost-abelian Lie algebras, classifying those that admit a left-invariant second-Chern--Einstein metric with a parallel non-zero Lee form.

\end{abstract}

\maketitle

\section{introduction}
    An almost-Hermitian manifold $(M,g,J)$ is a manifold $M$ equipped with a Riemannian metric $g$ and a $g$-orthogonal almost-complex structure $J.$ The almost-Hermitian structure $(g,J)$ induces the fundamental $2$-form $F(\cdot,\cdot)=g(J\cdot,\cdot).$ The Lee form $\theta$ associated to the almost-Hermitian structure $(g,J)$ is defined as
    $$dF^{n-1}=\frac{1}{n-1}\theta\wedge F^{n-1},$$
    where $d$ is the exterior derivative and $2n$ is the real dimension of the manifold. The almost-Hermitian metric $g$ (with a unit total volume) is called Gauduchon if $\delta^g\theta=0$, where $\delta^g$ is the adjoint of $d$ with respect to $g$, and almost-K\"ahler if $dF=0$. On the other hand, if the almost-complex structure $J$ is integrable then the pair $(g,J)$ is a Hermitian structure and it is K\"ahler if it is almost-K\"ahler. A $4$-dimensional almost-Hermitian manifold $(M,g,J)$ is locally conformally symplectic (LCS) if $d\theta=0$, while in higher dimension the LCS condition becomes $dF=\theta\wedge F$~\cite{MR418003,MR809073} (for a general introduction to the subject see~\cite{MR3880223}). In the integrable case, LCS is locally conformally K{\"a}hler (LCK).

    In the Hermitian and almost-Hermitian geometry there are some natural connections other then $D^g$ the Levi--Civita one. Among these we have the Chern connection $\nabla$~\cite{MR66733,MR0165458,MR1456265} defined as the unique connection preserving the almost-Hermitian structure and having a $J$-anti-invariant torsion; the Bismut connection $\nabla^B$~\cite{MR1456265} which is, in the integrable case, the unique connection preserving the Hermitian structure and having a skew-symmetric torsion; and the canonical Weyl connection $D^W$ defined as the unique torsion-free connection such that $D^Wg=\theta\otimes g$. To any of these connections can be associated many different Einstein-type equations. See for example~\cite{MR1228613, MR1249377,MR1272308,MR1199072,MR1200290,MR1363825, MR1798617,MR1726786, MR2447177, Barco:2022ve} for the Einstein--Weyl problem and~\cite{MR2673720} for the Bismut--Einstein problem in the Hermitian case. 
    
    In this note we focus on the {\em Second-Chern--Einstein problem}, which is stated as follows (for more details see~\cite{MR571563,MR633563,MR1477631,MR4125707}). We define the second Chern--Ricci form $r$ as $$r=R^\nabla(F),$$ that is, the image of the fundamental form by the Chern curvature $R^\nabla$.
    \begin{defn*}
        Given an almost-Hermitian manifold $(M,g,J)$ of real dimension $2n$, the metric $g$ is said to be second-Chern--Einstein if 
        $$ r = \frac{tr_F \, r}{n}F.$$
    \end{defn*}
    In the Hermitian case, second-Chern--Einstein metrics were studied in~\cite{MR571563,MR1477631,MR2781927,MR2795448,MR2925476,MR3869430,MR4125707,Angella:2020vt}. The condition of being second-Chern--Einstein is conformally invariant (even in the non-integrable case~\cite{MR4184828}). Since $D^W$ is invariant under conformal change of the metric, that led to explore the relation between Einstein--Weyl metrics and second-Chern--Einstein metrics. It turns out that in real dimension $4$, in the Hermitian case, the Einstein--Weyl condition and the second-Chern--Einstein condition (as well as the Bismut--Einstein condition) are equivalent~\cite[Theorem 1]{MR1477631} and so the only compact Hermitian non-K\"ahler $4$-manifold admitting a second-Chern--Einstein metric is the Hopf surface~\cite{MR1477631}. Even if this equivalence is no longer true for almost-Hermitian structures some crucial relations persist, see for example Corollary~\ref{cor_2-chern} and Proposition~\ref{prop: second chern first bismut}.

    In this note, we investigate the existence of second-Chern--Einstein metrics on $4$-dimensional almost-Hermitian manifolds. In Section~\ref{Sec_2}, we collect some general results about the geometry of these manifolds. In particular, we observe the following (for the analog in the Einstein--Weyl case see~\cite{MR1171560})
    \begin{lem*}[Lemma~\ref{killing}]
        Let $(M,\tilde{g},J)$ be a compact $4$-dimensional almost-Hermitian manifold and $\tilde{g}$ is a second-Chern--Einstein metric. Suppose that the Gauduchon metric $g$ in the conformal class $[\tilde{g}]$ satisfies $(D^g\theta)^{sym,J,-}=0$, where $D^g$ is the Levi--Civita connection of $g$ and $\theta$ is the Lee form of $(g,J)$. Then, the $g$-Riemannian dual of $\theta$ is a Killing vector field of $g$. Here $(\cdot)^{sym,J,-}$ denotes the $g$-symmetric $J$-anti-invariant part.
    \end{lem*}
    Taking twice the trace of the curvature of $D^W$, we get the conformal scalar curvature $s^W$, which enters the pictures through the following
    \begin{thm*}[Theorem~\ref{thm1}]
        Suppose that $(M,\tilde{g},J)$ is a $4$-dimensional compact locally conformally symplectic manifold and $\tilde{g}$ is a second-Chern--Einstein metric. Suppose that the Gauduchon metric $g$ in the conformal class $[\tilde{g}]$ satisfies $(D^g\theta)^{sym,J,-}=0$, where $\theta$ is the Lee form of $(g,J)$. Then, either
        \begin{enumerate}
            \item $(M,g,J)$ is a second-Chern--Einstein almost-K\"ahler manifold, or\\
            \item $\theta$ is $D^g$-parallel and the conformal scalar curvature $s^W$ is non-positive. Moreover, $s^W$ is identically zero if and only if $J$ is integrable and so $(M,J)$ is a Hopf surface as described in~\cite[Theorem 2]{MR1477631}. Furthermore, if $s^W$ is nowhere zero then $\chi=\sigma=0,$ where $\chi$ and $\sigma$ are the Euler class and signature of $M$ respectively.
        \end{enumerate}
    \end{thm*}

    In Section~\ref{Sec_3}, we give examples of compact $4$-dimensional second-Chern--Einstein almost-Hermitian manifolds. Some examples are locally conformally symplectic and some satisfy the condition $(D^g\theta)^{sym,J,-}=0$. We also remark that in those examples the second-Chern--Einstein metrics may have positive or zero Chern scalar curvature. It is observed that in the integrable case (in higher dimension) second-Chern--Einstein Hermitian non-K\"ahler metrics with negative Chern scalar curvature are still missing (see~\cite{Angella:2020vt}).
    
    Finally, in Section~\ref{Sec_4}, we study the Bismut--Einstein and the second-Chern--Einstein problems on $4$-dimensional almost-abelian Lie algebras. Using these class of manifolds, we show the existence in the almost-Hermitian case of Bismut--Einstein metrics with $dJdF=0$, while we recall that in the Hermitian non-K\"ahler case, examples of this kind are still missing. We also give a classification of $4$-dimensional unimodular almost-abelian Lie algebras equipped with a left-invariant almost-Hermitian non-Hermitian second-Chern--Einstein metric such that $\theta$ is $D^g$-parallel and non-zero.
    \begin{thm*}[Theorem \ref{thm: classification}]
     Let $\g$ be a $4$-dimensional unimodular almost-abelian Lie algebra equipped with a left-invariant almost-Hermitian non-Hermitian structure $(g,J)$ such that the Lee form $\theta$ is $D^g$-parallel and non-zero. Suppose
that $(g,J)$ is a solution to the second-Chern--Einstein problem. Then $\g$ is isomorphic to one of the following Lie algebras
\begin{enumerate}
\item $\mathcal{A}_{3,6}\oplus \mathcal{A}_1:\,[e_1,e_3]=-e_2,\quad [e_2,e_3]=e_1$.\\
\item $\mathcal{A}_{3,4}\oplus \mathcal{A}_1:\,[e_1,e_3]=e_1,\quad [e_2,e_3]=-e_2.$\\
\end{enumerate}
Both Lie algebras admit compact quotients.

          \end{thm*}
    Here we use the same notation of Lie algebras as~\cite{MR404362} (for a classification of locally conformally symplectic Lie algebras see~\cite{MR3763412,MR4088745}).

\section{Preliminaries}\label{Prel}



    In all the following $(M,g,J)$ will be an almost-Hermitian manifold of real dimension 4. We denote by $D^g$ the Levi-Civita connection associated to the Riemannian metric $g$ and by $R^\nabla,R^W,R^g,R^B$ the curvature tensors of $\nabla,D^W,D^g,\nabla^B$ respectively, where we use the following convention $R^\nabla_{X,Y}=\nabla_{[X,Y]}-[\nabla_X,\nabla_Y],$ etc. Moreover, given a $2$-tensor $\psi$, we denote by $\psi^{J,+}$ its $J$-invariant part, $\psi^{J,-}$ its $J$-anti-invariant part, $\psi^{sym}$ its $g$-symmetric part and $\psi^{anti-sym}$ its $g$-anti-symmetric part. Moreover, a $2$-form $\phi$ can be decomposed into a $g$-orthogonal sum $\phi=\phi^++\phi^-$, where $\phi^+$ is self dual i.e. $\ast_g\phi^+=\phi^+$ and $\phi^-$ is anti-self dual $\ast_g\phi^-=-\phi^-$ under the action of the Riemannian Hodge operator $\ast_g.$
    
    It follows from~\cite{MR1456265} that the Chern connection $\nabla$ is related to the Levi--Civita connection $D^g$ by
    \begin{equation}\label{chern_con}
        \nabla_XY=D^g_XY-\frac{1}{2}\theta(JX)JY-\frac{1}{2}\theta(Y)X+\frac{1}{2}g(X,Y)\theta^\sharp+g(X,N(\cdot,Y)),
    \end{equation}
    where $\sharp$ is the isomorphism induced by the metric g between $1$-forms and vector fields. Moreover, the canonical Weyl connection $D^W$ is related to the Levi--Civita connection $D^g$ by
    \begin{equation}\label{weyl_con}
        D^W_XY=D^g_XY-\frac{1}{2}\theta(X)Y-\frac{1}{2}\theta(Y)X+\frac{1}{2}g(X,Y)\theta^\sharp,
    \end{equation}
    
    We also remark that (for more details see~\cite{MR1727302})
    \begin{equation}\label{int_weyl}
        g(\left(D^W_ZJ\right)X,JY)=-2g(N(X,Y),Z),
    \end{equation}
    where $N$ is the Nijenhuis tensor defined by
    $$4N(X,Y)=[JX,JY]-[X,Y]-J[JX,Y]-J[X,JY].$$ 
    Moreover, on any $4$-dimensional almost-Hermitian manifold, Gauduchon proved~\cite[Proposition 1]{MR1456265} that
    \begin{equation}\label{cycle_n}
        g(N(X,Y),Z)+g(N(Y,Z),X)+g(N(Z,X),Y)=0,
    \end{equation}
    for any vector fields $X,Y,Z.$

    The almost-complex structure $J$ is integrable if and only if $N$ vanishes~\cite{MR88770}. Hence, $D^W$ preserves $J$ if and only if $J$ is integrable.
    

    We will now list some of the most relevant curvatures that can be obtained tracing the curvature tensors $R^\nabla,R^W,R^g,R^B$ together with their relations. To define them we will consider a $J$-adapted $g$-orthonormal frame of the tangent bundle $\{e_1,e_2=Je_1,e_3,e_4=Je_3\}$. 

    The first Chern--Ricci form $\rho^\nabla$ (called also the Hermitian Ricci form) of the Chern connection $\nabla$ is defined by 
    \begin{equation*}
        \rho^\nabla(X,Y)=\frac{1}{2}\sum_{i=1}^4g(R^\nabla_{X,Y}e_i,Je_i),
    \end{equation*}
    similarly, the Bismut--Ricci form $Ric^B$ is defined by
    \begin{equation*}
        Ric^B(X,Y)=\frac{1}{2}\sum_{i=1}^4g(R^B_{X,Y}e_i,Je_i).
    \end{equation*}
    These $2$-forms $\rho^\nabla$ and $Ric^B$ are closed and they are representatives of the first Chern class $2\pi c_1(TM,J)$ in De Rham cohomology. Indeed, they differ by the exact factor $dJ\theta$, i.e.
    \begin{equation}\label{eq: ricci chern and bismut}
        Ric^B = \rho^\nabla + dJ\theta,
    \end{equation}
    see~\cite[Equation (2.7.6)]{MR1456265} (see also~\cite{MR1836272,Barbaro:2021tt} and the references therein).
    We also remark that these forms are not necessarily $J$-invariant.
   
    We denote by $r$ the second Chern--Ricci form of $\nabla$ defined by
    \begin{equation*}
        r(X,Y)=\frac{1}{2}\sum_{i=1}^4g(R^\nabla_{e_i,Je_i}X,Y),
    \end{equation*}
    or equivalently, $$r=R^\nabla(F).$$ $r$ is a $J$-invariant $2$-form but not closed in general. 

    Similarly, $R^W(F)$ is given by the formula $$R^W(F)=\frac{1}{2}\sum_{i=1}^4g(R^{W}_{e_i,Je_i}X,Y).$$ 
    The tensor $R^W(F)$ is a $2$-form and it is not $J$-invariant in general. However, when $J$ is integrable, $D^W$ preserves $J$ and so $R^W(F)=(R^W(F))^{J,+}$ is $J$-invariant. The Weyl Ricci tensor $Ric^W$ is defined in~\cite{MR1477631} as
    $$Ric^W(X,Y)=\sum_{i=1}^4g(R^W_{e_i,X}e_i,Y).$$
    Note that the tensor $Ric^W$ is symmetric (this is only true in dimension $4$). On the other hand,  the tensor $\widetilde{Ric}^W$, defined as (see for example~\cite{MR1199072}) $$\widetilde{Ric}^W(X,Y)=\sum_{i=1}^4g(R^W_{X,e_i}Y,e_i),$$
    is not symmetric in general. Its anti-symmetric part is $d\theta$ while its symmetric part is $Ric^W$, that is~(see\cite{MR1975033}) 
    $$\widetilde{Ric}^W=Ric^W+d\theta.$$

    We define the Riemannian Ricci tensor $Ric^g$ as $$Ric^g(X,Y)=\sum_{i=1}^4g(R^g_{e_i,X}e_i,Y),$$ 
    and the $\star$-Ricci tensor $\rho^\star$ as $$\rho^\star(X,Y)=R^g(F)(X,JY)=\frac{1}{2}\sum_{i=1}^4g(R^g_{e_i,Je_i}X,JY).$$
    From the definition, it follows that $\rho^\star(X,Y)=\rho^\star(JY,JX)$ so $\rho^\star$ is symmetric if and only it is $J$-invariant.
    
    We define the Hermitian scalar curvature $s^H$ as $$s^H=\sum_{i=1}^4r(e_i,Je_i),$$ the Riemannian scalar curvature $s^g$ as $s^g=\sum_{i=1}^4Ric^g(e_i,e_i),$ the conformal scalar curvature $s^W$ as $s^W=\sum_{i=1}^4Ric^W(e_i,e_i),$ and the $\star$-scalar curvature $s^\star$ as $s^\star=\sum_{i=1}^4\rho^\star(e_i,e_i).$

    The conformal scalar curvature $s^W$ is then related to the Riemannian scalar curvature $s^g$ by~(see for example~\cite{MR1975033})
    \begin{equation}\label{conformal_sca}
        s^W=s^g-3\,\delta^g\theta-\frac{3}{2}\|\theta\|^2,
    \end{equation}
    where $\|\theta\|^2=g(\theta,\theta).$
    Furthermore, for any almost-Hermitian manifold of dimension $4$, we have that~\cite{MR626479}
    \begin{equation}\label{diff_ricci}
        \left(\rho^\star\right)^{sym}-\left(Ric^g\right)^{J,+}=\frac{s^\star-s^g}{4}g.
    \end{equation}
    On the other hand, $s^\star$ is related to the Riemannian scalar curvature $s^g$ by (see~\cite{MR945613, MR2604242}, while in the integrable case~\cite{MR696036})
    \begin{equation}\label{diff_sca}
    s^\star-s^g=-2\delta^g\theta-\|\theta\|^2+2\|N\|^2.
    \end{equation}

\section{Second-Chern--Einstein metrics}\label{Sec_2}
    Let $(M,g,J)$ be an almost-Hermitian manifold of real dimension $4$. We recall that with our notations the metric $g$ is said to be second-Chern--Einstein if  
    $$r=\frac{s^H}{4} F,$$ 
    (we note that $s^H$ is not necessarily constant here). 

\subsection{Relation between second Chern--Ricci and Weyl--Ricci tensors}     
    Under a conformal change $\tilde{g}=e^{2f}g$, the conformal variation of the second Chern--Ricci form $r$ of $(g,J)$ is given by ~(\cite{MR4184828}, in the Hermitian case see~\cite{MR742896})
    \begin{equation}\label{conf-change}
        \tilde{r}=r+\left(\Delta^gf+g(\theta,df)\right)F,
    \end{equation}
    where $\tilde{r}$ is the second Ricci form of $(\tilde{g},J)$ and $\Delta^g$ is the Riemannian Laplacian of $g$. The condition of being second-Chern--Einstein, in the almost-Hermitian setting, is then conformally invariant. Hence, we investigate the relation between the curvature $R^\nabla$ of the Chern connection $\nabla$ and the curvature $R^W$ of the Weyl connection $D^W$.

    \begin{prop}\label{relation-chern-weyl}
        The curvatures $R^\nabla$ and $R^W$ are related by
        \begin{align*}
            R^\nabla_{X,Y}Z = R^W_{X,Y}Z & - \frac{1}{2}(dJ\theta)(X,Y)JZ-\frac{1}{2}(d\theta)(X,Y)Z \\
            & -\frac{1}{2}\left(D^W_X\left(D^W_YJ\right)-D^W_Y\left(D^W_XJ\right)-D^W_{[X,Y]}J \right)JZ\\
            & +\frac{1}{4}\left(\left(D^W_XJ\right)\left(D^W_YJ\right) -\left(D^W_YJ\right)\left(D^W_XJ\right)   \right)Z.
        \end{align*}
    \end{prop}
    \begin{proof}
        From~(\ref{weyl_con}) and~(\ref{chern_con}), we obtain that
        $$\nabla_XY=D^W_XY+\frac{1}{2}\theta(X)Y-\frac{1}{2}\theta(JX)JY+g(X,N(\cdot,Y)).$$
        From~(\ref{int_weyl}), we get that
        $$\nabla_XY=D^W_XY+\frac{1}{2}\theta(X)Y-\frac{1}{2}\theta(JX)JY+\frac{1}{2}(D^W_XJ)JY.$$
        Then, we compute the curvature.
    \end{proof}
    When $J$ is integrable, the relation reduces to (see~\cite{MR1477631}) $$R^\nabla=R^W-\frac{1}{2}(dJ\theta)\otimes J-\frac{1}{2}(d\theta)\otimes Id.$$

    We remark that the part $\left(R^W_{X,Y}\right)^{J,-}$ of $R^W_{X,Y}$ that anticommutes with $J$ is given precisely by (see for example~~\cite[Equation (2.15)]{MR1727302})
    \begin{equation}\label{weyl_anti}
        \left(R^W_{X,Y}\right)^{J,-}=\frac{1}{2}\left(D^W_X\left(D^W_YJ\right)-D^W_Y\left(D^W_XJ\right)-D^W_{[X,Y]}J \right).
    \end{equation}
    We obtain then the following
    \begin{cor}\label{cor1}
        Let $(M,g,J)$ be an almost-Hermitian $4$-dimensional manifold. Then,
        \begin{align*}
            R^\nabla_{X,Y}Z=\left(R^W_{X,Y}\right)^{J,+}Z & -\frac{1}{2}(dJ\theta)(X,Y)JZ-\frac{1}{2}(d\theta)(X,Y)Z\\
            & +\frac{1}{4}\left(\left(D^W_XJ\right)\left(D^W_YJ\right) -\left(D^W_YJ\right)\left(D^W_XJ\right)   \right)Z,
        \end{align*}
        where $\left(R^W_{X,Y}\right)^{J,+}$ is the part of $R^W_{X,Y}$ that commutes with $J.$
    \end{cor}

    It follows from Corollary~\ref{cor1} that we can relate the second Chern--Ricci curvature with $R^W(F)$ extending the relation in the integrable case (see~\cite[Theorem 1]{MR1477631})
    \begin{cor}\label{cor_2-chern}
        Let $(M,g,J)$ be an almost-Hermitian $4$-dimensional manifold. Then,
        \begin{equation*}
            r=(R^W(F))^{J,+}+\frac{1}{2}(\delta^g\theta+\|\theta\|^2)F-\frac{1}{4}\|N\|^2F.
        \end{equation*}
        In particular, $g$ is second-Chern--Einstein if and only if $(R^W(F))^{J,+}$ is a multiple of $F.$
    \end{cor}
    \begin{proof}
        We first have that $g(dJ\theta,F)=-\delta^g\theta-\|\theta\|^2,$ and $g(d\theta,F)=0.$ Moreover, using~(\ref{int_weyl}), we see that 
        \begin{align*}
            \sum_{i=1}^4g\left(\left(D^W_{e_i}J\right)\left(D^W_{Je_i}J\right) -\left(D^W_{Je_i}J\right)\left(D^W_{e_i}J\right)Z,V\right)&=-2\sum_{i=1}^4g\left(\left(D^W_{e_i}J\right)\left(D^W_{e_i}J\right)Z,JV\right),\\
            &=2\sum_{i=1}^4g\left(\left(D^W_{e_i}J\right)Z,\left(D^W_{e_i}J\right)JV\right),\\
            &=4\sum_{i=1}^4g\left(N\left(Z,\left(D^W_{e_i}J\right)JV  \right),Je_i\right),\\
            &=4\sum_{i,j=1}^4g(N\left(Z,e_j  \right),Je_i) g\left(\left(D^W_{e_i}J\right)JV,e_j\right),\\
            &=8\sum_{i,j=1}^4g(N\left(Z,e_j  \right),Je_i) g\left(N(JV,e_j),Je_i\right),\\
            &=-8\sum_{i,j=1}^4g(N\left(JZ,e_j  \right),e_i) g\left(N(V,e_j),e_i\right).\\
        \end{align*}
        Now, we use the crucial fact that we are in dimension 4 with $\|N\|^2=8\|N(e_1,e_3)\|^2$ 
        \begin{align*}
            \sum_{i,j=1}^4g(N\left(JZ,e_j  \right),e_i) g\left(N(V,e_j),e_i\right)&=\sum_{i,j,k,l=1}^4g(JZ,e_k)g(V,e_l)g(N\left(e_k,e_j  \right),e_i) g\left(N(e_l,e_j),e_i\right),\\
            &=\sum_{j,k,l=1}^4g(JZ,e_k)g(V,e_l)g(N\left(e_k,e_j  \right),N(e_l,e_j)),\\
            &=\sum_{j,k=1}^4g(JZ,e_k)g(V,e_k)g(N\left(e_k,e_j  \right),N(e_k,e_j)),\\
            &=\frac{1}{4}\|N\|^2g(JZ,V).
        \end{align*}
        The result then follows.
    \end{proof}

    The canonical Weyl connection $D^W$ is said to be Einstein--Weyl if $Ric^W$ (or equivalently the symmetric part of $\widetilde{Ric}^W$) is proportional to the metric $g.$ When $J$ is integrable, the metric $g$ is second-Chern--Einstein if and only if $D^W$ is Einstein--Weyl~\cite[Theorem 1]{MR1477631}. This is due to the fact that $Ric^W(JX,Y)=R^W(F)(X,Y)$ because $R^W_{X,Y}$ commutes with $J$.
    
    From~(\ref{weyl_con}) we also obtain
    \begin{cor}\label{cor_ric-weyl}
        Let $(M,g,J)$ be an almost-Hermitian $4$-dimensional manifold. Then,
        \begin{equation*}
            R^W(F)^{J,+}_{\cdot,J\cdot}=(\rho^\star)^{sym}+\left(D^g\theta\right)^{sym,J,+}-\frac{1}{4}\|\theta\|^2g+\frac{1}{2}\left(\theta\otimes\theta\right)^{J,+}.
        \end{equation*}
    \end{cor}
    
    Then, combining Corollary~\ref{cor_2-chern} with Corollary~\ref{cor_ric-weyl} and using~(\ref{diff_ricci}) and~(\ref{diff_sca}), we deduce the following
    \begin{cor}\label{chern-ric}
        Let $(M,g,J)$ be an almost-Hermitian $4$-dimensional manifold. Then,
        \begin{equation*}
            r_{\cdot,J\cdot}=\left(Ric^g\right)^{J,+}+\frac{1}{4}\|N\|^2g+\left(D^g\theta\right)^{sym,J,+}+\frac{1}{2}\left(\theta\otimes\theta\right)^{J,+}.
        \end{equation*}
    \end{cor}
    
\subsection{The conditions $(D^g\theta)^{sym,J,-}=0$ and $\mathcal{L}_{\theta^\sharp}g=0$}
    In~\cite{MR633563} Gauduchon proved that every conformal class of an almost-Hermitian metric has a unique representative $g$ (up to constant) which satisfies $\delta^g\theta=0$. We call such metric (with a unit total volume) a Gauduchon metric. With this assumption, using Corollory~\ref{chern-ric} we can prove the following (when the metric is Einstein--Weyl see~\cite{MR1171560} and also~\cite{MR1363825})
    \begin{lem}\label{killing}
        Let $(M,\tilde{g},J)$ be a compact $4$-dimensional almost-Hermitian manifold and $\tilde{g}$ is a second-Chern--Einstein metric. Suppose that the Gauduchon metric $g$ in the conformal class $[\tilde{g}]$ satisfies $(D^g\theta)^{sym,J,-}=0$, where $\theta$ is the Lee form of $g$. Then, the $g$-Riemannian dual of $\theta$ is a Killing vector field of $g.$
    \end{lem}
    \begin{proof}
        The condition $r=\frac{s^H}{4} F$ is conformally invariant. Then, from Corollory~\ref{chern-ric}, we have for the Gauduchon metric $g$
        \begin{align*}
            \frac{s^H}{4} g&=(Ric^g)^{J,+}+\frac{1}{4}\|N\|^2g+(D^g\theta)^{sym,J,+}+\frac{1}{2}(\theta\otimes\theta)^{1,1}.
        \end{align*}
        We recall that $\delta^g\theta=-g(D^g\theta,g)$. Taking the inner product with $(D^g\theta)^{sym}=(D^g\theta)^{sym,J,+}$ and integrating we have
        \begin{multline*}
            -\int_M\frac{s^H}{4}\delta^g\theta\, v_g=\int_Mg((Ric^g)^{J,+},(D^g\theta)^{sym})\\
            -\frac{1}{4}\|N\|^2\delta^g\theta+\|(D^g\theta)^{sym}\|^2+\frac{1}{2}g((\theta\otimes\theta)^{J,+},(D^g\theta)^{sym})\,v_g,
        \end{multline*}
        where $v_g$ is the volume form. Since the metric is Gauduchon i.e. $\delta^g\theta=0$, we obtain
        \begin{align*}
            \int_M\|(D^g\theta)^{sym}\|^2\,v_g&=-\int_Mg((Ric^g)^{J,+},(D^g\theta)^{sym})+\frac{1}{2}g((\theta\otimes\theta)^{J,+},(D^g\theta)^{sym})\,v_g,\\
            &=-\int_Mg(Ric^g,D^g\theta)+\frac{1}{2}g(\theta\otimes\theta,D^g\theta)\,v_g,\\
            &=-\int_Mg(\delta^gRic^g,\theta)+\frac{1}{2}g(\delta^g(\theta\otimes\theta),\theta)\,v_g,\\
            &=-\int_Mg(-\frac{1}{2}ds^g,\theta)+\frac{1}{2}\left(\delta^g\theta \|\theta^g\|^2-g(D^g_{\theta}\theta,\theta)\right)\,v_g,\\
            &=-\int_M-\frac{1}{2}s^g\delta^g\theta+\frac{1}{2}\left(-\frac{1}{2}g(d\|\theta\|^2_g,\theta)\right)\,v_g,\\
            &=\frac{1}{4}\int_M\|\theta\|^2_g\delta^g\theta\,v_g =0, 
        \end{align*}
        where we used the contracted Bianchi identity $\delta^gRic^g=-\frac{1}{2}ds^g$. The Lemma follows.
    \end{proof}

    We have seen that in the integrable case second-Chern--Einstein is equivalent to Weyl--Einstein (also in greater dimension if we ask the LCK condition). Then, second-Chern--Einstein together with LCK implies that $D^g\theta = 0$ and that both $\theta^\sharp, J\theta^\sharp$ are Killing real holomorphic vector fields, see~\cite{MR1363825,MR1249377}. However, in the almost-Hermitian case the condition $(D^g\theta)^{sym,J,-}=0$ is necessary, see the examples in Section~\ref{Sec_3}. \\
    Moreover, we remark that $(D^g\theta)^{sym,J,-}=0$ is equivalent to $\left(\mathcal{L}_{\theta^\sharp}J\right)^{sym}=0,$ where $\mathcal{L}$ is the Lie derivative, so the flow of the vector field $\theta^\sharp$ does not necessarily preserve $J$ i.e. $\theta^\sharp$ is not necessarily a real holomorphic vector field. We also remark that when $M$ is compact, the condition that $J\theta^\sharp$ is a Killing vector field implies that $(M,g,J)$ is LCS i.e. $d\theta=0.$ Indeed, applying the Lie derivative $\mathcal{L}_{J\theta^\sharp}$ to the relation $F=g(J\cdot,\cdot)$ we get 
    $$d\theta=-2\left(D^gJ\theta \right)^{sym}_{J\cdot,\cdot}-g\left(\mathcal{L}_{J\theta^\sharp}J\cdot,\cdot\right).$$
    Hence, if $\left(D^gJ\theta \right)^{sym}=0$ then $d\theta$ is $J$-anti-invariant. Thus, $d\theta$ is a self-dual $d$-exact $2$-form on a compact manifold so $d\theta=0.$ However, the converse is not true in general: if $d\theta=0$ then only the $J$-invariant part of $\left(D^gJ\theta \right)^{sym}$ vanishes i.e. $\left(D^gJ\theta \right)^{sym,J,+}=0$ so $J\theta^\sharp$ is not necessarily a Killing vector field.

\subsection{Constant scalar curvatures}
    Let $g$ be a second-Chern--Einstein metric. If the Hermitian scalar curvature $s^H$ is a constant (respectively a non-constant function) then $g$ is called a strong (respectively weak) second-Chern--Einstein metric~\cite{MR4125707}. Since the conformal change of the second Chern--Ricci form~(\ref{conf-change}) is the same as in the integrable case, we can generalize~\cite[Theorem B]{MR4125707} to the almost-Hermitian setting (see also~\cite{MR633563,MR3696598,MR3833814,MR4070351,MR4203643})
    \begin{defn}
        Let $g$ be the Gauduchon metric in the conformal class $[\tilde{g}]$. Then, the fundamental constant~\cite{MR486672,MR779217,MR1712115} is  $$C(M,[\tilde{g}],J)=\int_M\,s^H\frac{F^n}{n!},$$
        where $s^H$ is the Hermitian scalar curvature of the almost-Hermitian structure $(g,J)$ inducing the fundamental form $F.$
    \end{defn}

    \begin{thm}~\cite[Theorem B]{MR4125707}\cite[Corollary 5.10]{MR4184828}
        Let $(M,\tilde{g},J)$ be a compact $4$-dimensional almost-Hermitian manifold and suppose that $\tilde{g}$ is a weak second-Chern--Einstein metric. Then, there is a representative in $[\tilde{g}]$ such that its Hermitian scalar curvature has the same sign as $C(M,[\tilde{g}],J)$. Moreover, if $C(M,[\tilde{g}],J)\leq 0$, then there is a strong second-Chern--Einstein representative in $[\tilde{g}]$.
    \end{thm}

    Moreover, we can prove that $s^W$ is constant under some conditions (in the Einstein--Weyl case, see~\cite[Proposition 2.1]{MR1199072}). First of all, thanks to Corollary~\ref{chern-ric} and Lemma~\ref{killing}, we get the following
    \begin{lem}\label{expression-ricci}
        Suppose that $(M,\tilde{g},J)$ is a compact $4$-dimensional almost-Hermitian manifold where $\tilde{g}$ is a second-Chern--Einstein metric. Suppose that the Gauduchon metric $g$ in the conformal class $[\tilde{g}]$ satisfies $(D^g\theta)^{sym,J,-}=0$, where $\theta$ is the Lee form of $g$. Then
        \begin{align}\label{ric-expression}
            (Ric^g)^{J,+}&=\frac{s^H}{4} g-\frac{1}{4}\|N\|^2g-\frac{1}{2}(\theta\otimes\theta)^{J,+}.
        \end{align}
    \end{lem}
    Then we can prove that
    \begin{prop}
        Let $(M,\tilde{g},J)$ be a compact $4$-dimensional almost-Hermitian manifold and $\tilde{g}$ is a second-Chern--Einstein metric. Suppose that the Gauduchon metric $g$ in the conformal class $[\tilde{g}]$ satisfies $(D^g\theta)^{sym,J,-}=0$, where $\theta$ is the Lee form of $g$. Moreover, suppose that $Ric^g$ is $J$-invariant and $J\theta^\sharp$ is a Killing vector field. Then the conformal scalar curvature $s^W$ of $g$ is constant.
    \end{prop}
    \begin{proof}
        We apply the codifferential $\delta^g$ to~(\ref{ric-expression}) and we use the contracted Bianchi identity $\delta^gRic^g=-\frac{1}{2}ds^g$. We obtain
        \begin{align*}
            -\frac{1}{2}ds^g&=-\frac{1}{4}ds^H+\frac{1}{4}d(\|N\|^2)-\frac{1}{4}\delta^g\left(\theta\otimes\theta \right)-\frac{1}{4}\delta^g\left(J\theta\otimes J\theta \right),\\
            &=-\frac{1}{4}ds^H+\frac{1}{4}d(\|N\|^2)-\frac{1}{4}\left(\delta^g\theta\otimes\theta-D^g_{\theta^\sharp}\theta \right)-\frac{1}{4}\left(\delta^gJ\theta\otimes\theta-D^g_{J\theta^\sharp}J\theta  \right),\\
            &=-\frac{1}{4}ds^H+\frac{1}{4}d(\|N\|^2)+\frac{1}{4}D^g_{\theta^\sharp}\theta+\frac{1}{4}D^g_{J\theta^\sharp}J\theta.
        \end{align*}
        Hence
        $$ds^g=\frac{1}{2}ds^H-\frac{1}{2}d(\|N\|^2)-\frac{1}{2}D^g_{\theta^\sharp}\theta-\frac{1}{2}D^g_{J\theta^\sharp}J\theta.$$
        On the other hand, from the trace of~~(\ref{ric-expression}), we have $s^g=s^H-\|N\|^2-\frac{1}{2}\|\theta\|^2.$ Thus,
        \begin{equation*}
            d\left(s^H-\|N\|^2 \right)=d(\|\theta\|^2)-D^g_{\theta^\sharp}\theta-D^g_{J\theta^\sharp}J\theta.
        \end{equation*}
        Applying again the codifferential $\delta^g$, we have
        \begin{equation}\label{sh_constant}
            \Delta^g\left(s^H-\|N\|^2 \right)=\Delta^g(\|\theta\|^2)-\delta^gD^g_{\theta^\sharp}\theta-\delta^gD^g_{J\theta^\sharp}J\theta,
        \end{equation}
        where $\Delta^g$ is the Laplacian. From Lemma~\ref{killing}, $\theta^\sharp$ is Killing. Thus, using Cartan formula and because $\mathcal{L}_{\theta^\sharp}\theta=0$ we get
        \begin{equation}
            \delta^gD^g_{\theta^\sharp}\theta=\frac{1}{2}\delta^g\left(d\theta(\theta^\sharp,\cdot)\right)=-\frac{1}{2}\delta^gd(\|\theta\|^2)=-\frac{1}{2}\Delta^g(\|\theta\|^2).
        \end{equation}
        Similarly, because $J\theta^\sharp$ is a Killing vector field, we have $\delta^gD^g_{J\theta^\sharp}J\theta=-\frac{1}{2}\Delta^g(\|\theta\|^2).$ From~(\ref{sh_constant}), we obtain
        \begin{equation*}
            \Delta^g\left(s^H-\|N\|^2-2\|\theta\|^2\right)=0.
        \end{equation*}
        Since $M$ is compact and from~(\ref{conformal_sca}) we get that $s^H-\|N\|^2-2\|\theta\|^2=s^g-\frac{3}{2}\|\theta\|^2=s^W$ is a constant.
    \end{proof}


    We can also deduce from~(\ref{weyl_con}), Lemma~\ref{killing} and Lemma~\ref{expression-ricci} that
    \begin{cor}
        Suppose that $(M,\tilde{g},J)$  is a compact $4$-dimensional almost-Hermitian manifold where $\tilde{g}$ is a second-Chern--Einstein metric. Suppose that the Gauduchon metric $g$ in the conformal class $[\tilde{g}]$ satisfies $(D^g\theta)^{sym,J,-}=0$, where $\theta$ is the Lee form of $g$. Then
        \begin{align*}
            Ric^W&=Ric^g-\frac{1}{2}\left(\|\theta\|^2g-\theta\otimes\theta\right),\\
            &=\frac{s^W}{4}g+(Ric^g)^{J,-}+\frac{1}{2}(\theta\otimes\theta)^{J,-}.
        \end{align*}
        In particular, the metric $g$ is Einstein--Weyl if and only if
        \begin{align}\label{E-W-condition}
            (Ric^g)^{J,-}&=-\frac{1}{2}(\theta\otimes\theta)^{J,-}.
        \end{align}
    \end{cor}
    
\subsection{The main theorems}
    
    Combining the equations ~(\ref{diff_ricci}) and ~(\ref{diff_sca}) we obtain
    \begin{equation}\label{diff_ricci+sca}
        \left(\rho^\star\right)^{sym}-\left(Ric^g\right)^{J,+}=-\frac{1}{4}\left(2\delta^g\theta+\|\theta\|^2-2\|N\|^2\right)g.
    \end{equation}
    We have already computed $\left(Ric^g\right)^{J,+}$ in ~(\ref{ric-expression}); now we want to compute $\rho^\star\left(\theta^\sharp,\theta^\sharp\right)$ so that we can evaluate ~(\ref{diff_ricci+sca}) on $\left(\theta^\sharp,\theta^\sharp\right)$ and obtain a condition on the almost-Hermitian structure when $g$ is second-Chern--Einstein. We will assume that the Riemannian dual of $\theta$ is a Killing vector field.

    \begin{lem}\label{gen_star}
        Suppose that $(M,{g},J)$ is a $4$-dimensional compact almost-Hermitian manifold. Suppose that the Riemannian dual $\theta^\sharp$ of the Lee form $\theta$ is a Killing vector field. Then,
        \begin{equation*}
            \rho^\star(\theta^\sharp,X)=-\frac{1}{2}\left(d\theta\right)^{J,-}(\theta^\sharp,X)+g\left(d\theta,N_{X}\right),
        \end{equation*}
        for any vector field $X$, where $N_{X}(Y,Z)=g(N(Y,Z),X).$
    \end{lem}
    \begin{proof}
        Let $\alpha$ be a $1$-form. Then
        \begin{align*}
            \left(\delta^g(D^g\alpha)^{J,+}-\delta^g(D^g\alpha)^{J,-}\right)(X)&=-\sum_{i=1}^4\left(D^g_{e_i}\left((D^g\alpha)^{J,+}-(D^g\alpha)^{J,-}\right)\right)(e_i,X),\\
            & \begin{multlined}
                =\sum_{i=1}^4 -D^g_{e_i}\left(D^g_{Je_i}\alpha(JX)\right) +D^g\alpha(JD^g_{e_i}e_i,JX)\\+D^g\alpha(Je_i,JD^g_{e_i}X),
            \end{multlined} \\
            & \begin{multlined}
                =\sum_{i=1}^4-D^g_{e_i}\left(D^g_{Je_i}\alpha(JX)\right) + D^g\alpha(D^g_{e_i}(Je_i),JX) \\ + D^g\alpha(Je_i,D^g_{e_i}(JX)) - D^g\alpha((D^g_{e_i}J)e_i,JX) \\ - D^g\alpha(Je_i,\left(D^g_{e_i}J\right)X),
                \end{multlined}\\
            & \begin{multlined}
                =\sum_{i=1}^4-\left(D^g_{e_i}\left(D^g_{Je_i}\alpha\right)\right)(JX) + D^g\alpha(D^g_{e_i}(Je_i),JX) \\ - D^g\alpha((D^g_{e_i}J)e_i,JX) - D^g\alpha(Je_i,\left(D^g_{e_i}J\right)X),
            \end{multlined}\\
            &=\sum_{i=1}^4\frac{1}{2}g(R^g_{e_i,Je_i}\alpha^\sharp,JX)-D^g\alpha(Je_i,\left(D^g_{e_i}J\right)X)-D^g\alpha(J\theta^\sharp,JX),\\
            &=\rho^\star(\alpha^\sharp,X)-D^g\alpha(J\theta^\sharp,JX)-\sum_{i=1}^4D^g\alpha(Je_i,\left(D^g_{e_i}J\right)X).
        \end{align*}
        On the other hand, using that $g(d\theta,F)=0$ and $\delta^g=-\ast_gd\,\ast_g$, we have
        \begin{align*}
            \delta^g(D^g\theta)^{J,+}-\delta^g(D^g\theta)^{J,-}&=\frac{1}{2}\delta^g(d\theta)^{J,+}-\frac{1}{2}\delta^g(d\theta)^{J,-},\\
            &=\frac{1}{2}\delta^g(d\theta)^{-}-\frac{1}{2}\delta^g(d\theta)^{+},\\
            &=-\frac{1}{2}\ast_gd\ast_g(d\theta)^{-}+\frac{1}{2}\ast_gd\ast_g(d\theta)^{+},\\
            &=\frac{1}{2}\ast_gd(d\theta)^{-}+\frac{1}{2}\ast_gd(d\theta)^{+},\\
            &=\frac{1}{2}\ast_gd\left((d\theta)^{-}+d(d\theta)^{+}\right)=\frac{1}{2}\ast_gd(d\theta)=0.
        \end{align*}
        We deduce then
        \begin{equation}\label{general-form}
            \rho^\star(\theta^\sharp,X)=D^g\theta(J\theta^\sharp,JX)+\sum_{i=1}^4D^g\theta(Je_i,\left(D^g_{e_i}J\right)X).
        \end{equation}
        Now, we would like to compute the second term in the right hand side of~(\ref{general-form}). We first recall that
        \begin{equation*}
            \left(D^g_XJ\right)Y=\frac{1}{2}g(X,Y)J\theta^\sharp+\frac{1}{2}\theta(JY)X+\frac{1}{2}g(JX,Y)\theta^\sharp-\frac{1}{2}\theta(Y)JX+2\left(g(N(Y,\cdot),JX)\right)^\sharp,
        \end{equation*}
        which can be easily deduced from~(\ref{weyl_con}) and~(\ref{int_weyl}). Hence,
        \begin{align}
            \sum_{i=1}^4D^g\theta(Je_i,\left(D^g_{e_i}J\right)X)=&\sum_{i=1}^4D^g_{Je_i}\theta\left(\frac{1}{2}g(e_i,X)J\theta^\sharp\right)+D^g_{Je_i}\theta\left(\frac{1}{2}\theta(JX)e_i\right)\nonumber\\
            &+D^g_{Je_i}\theta\left(\frac{1}{2}g(Je_i,X)\theta^\sharp\right)-D^g_{Je_i}\theta\left(\frac{1}{2}\theta(X)Je_i\right)+2g(N(X,D^g_{Je_i}\theta^\sharp),Je_i),\nonumber\\
            =&\,\frac{1}{2}g(D^g_{JX}\theta,J\theta)+\frac{1}{2}g(D^g_X\theta,\theta)+2\sum_{i=1}^4g(N(X,D^g_{Je_i}\theta^\sharp),Je_i),\label{eq-last}
        \end{align}
        where we use the fact that $g(d\theta,F)=\delta^g\theta=0.$ Moreover, using~(\ref{cycle_n}), we compute the third term in~(\ref{eq-last})
        \begin{align*}
            \sum_{i=1}^4g(N(X,D^g_{Je_i}\theta^\sharp),Je_i)&=\sum_{i,j=1}^4g(D_{Je_i}\theta^\sharp,e_j)g(N(X,e_j)Je_i),\\
            &=\sum_{i,j=1}^4-g(D^g_{Je_i}\theta^\sharp,e_j)g(N(e_j,Je_i),X)-(D^g_{Je_i}\theta^\sharp,e_j)g(N(Je_i,X),e_j),\\
            &=\sum_{i,j=1}^4-g(D^g_{Je_i}\theta^\sharp,e_j)g(N(e_j,Je_i),X)-(D^g_{e_j}\theta^\sharp,Je_i)g(N(X,Je_i),e_j).
        \end{align*}
        Thus,
        \begin{align*}
            \sum_{i=1}^4g(N(X,D^g_{Je_i}\theta^\sharp),Je_i)&=-\frac{1}{2}\sum_{i,j=1}^4g(D^g_{Je_i}\theta^\sharp,e_j)g(N(e_j,Je_i),X),\\
            &=\frac{1}{2}\sum_{i,j=1}^4g(D^g_{e_i}\theta^\sharp,e_j)g(N(e_i,e_j),X),\\
            &=\frac{1}{2}g(d\theta,N_X).
        \end{align*}
        From~(\ref{eq-last}), we deduce that
        \begin{align*}
            \sum_{i=1}^4D^g\theta(Je_i,\left(D^g_{e_i}J\right)X)&=\frac{1}{2}g(D^g_{JX}\theta,J\theta)+\frac{1}{2}g(D^g_X\theta,\theta)+g(d\theta,N_X).
        \end{align*}
        Finally, from~(\ref{general-form}) and~(\ref{eq-last}), we conclude that
        \begin{align*}
            \rho^\star(\theta^\sharp,X)&=D^g\theta(J\theta^\sharp,JX)+\frac{1}{2}g(D^g_{JX}\theta,J\theta)+\frac{1}{2}g(D^g_X\theta,\theta)+g(d\theta,N_X),\\
            &=\frac{1}{2}D^g\theta(J\theta^\sharp,JX)+\frac{1}{2}g(D^g_X\theta,\theta)+g(d\theta,N_X),\\
            &=-\frac{1}{2}(d\theta)^{J,-}(\theta^\sharp,X)+g(d\theta,N_X).
        \end{align*}
        The lemma follows.
    \end{proof}

    As consequences of Lemma~\ref{gen_star}, we obtain
    \begin{cor}\label{loc-conf-sym}
        Suppose that $(M,{g},J)$ is a $4$-dimensional compact locally conformally symplectic manifold. Suppose that the Riemannian dual $\theta^\sharp$ of the Lee form $\theta$ is a Killing vector field. Then,
        $$\rho^\star(\theta^\sharp,X)=0,$$
        for any vector field $X.$
    \end{cor}

    We also extend~\cite[Lemma 2]{MR1477631} to the almost-Hermitian setting.
    \begin{cor}\label{ric_star_lee}
        Suppose that $(M,{g},J)$ is a compact $4$-dimensional almost-Hermitian manifold. Suppose that the Riemannian dual $\theta^\sharp$ of the Lee form $\theta$ is a Killing vector field. Then, 
        \begin{equation*}
            \rho^\star(\theta^\sharp,\theta^\sharp)=g(d\theta,N_{\theta^\sharp}),
        \end{equation*}
        where $N_{\theta^\sharp}(X,Y)=g(N(X,Y),\theta^\sharp).$
    \end{cor}

    Notice that when $J$ is integrable $\rho^\star(\theta^\sharp,\theta^\sharp)=0.$ As a matter of fact, this condition simplifies in a natural way the problem, thus we will now prove our main theorems with the assumption that $d\theta = 0$ (Theorem~\ref{thm1}) or $N_{\theta^\sharp}=0$ (Theorem~\ref{N-theta}). 

    \begin{thm}\label{thm1}
        Suppose that $(M,\tilde{g},J)$ is a $4$-dimensional compact locally conformally symplectic manifold and $\tilde{g}$ is a second-Chern--Einstein metric. Suppose that the Gauduchon metric $g$ in the conformal class $[\tilde{g}]$ satisfies $(D^g\theta)^{sym,J,-}=0$, where $\theta$ is the Lee form of $g$. Then, either
        \begin{enumerate}
            \item $(M,g,J)$ is a second-Chern--Einstein almost-K\"ahler manifold, or\\
            \item $\theta$ is $D^g$-parallel and the conformal scalar curvature $s^W$ is non-positive and the $\star$-scalar curvature $s^\star$ is a positive constant. Moreover, $s^W$ is identically zero if and only if $J$ is integrable and so $(M,J)$ is a Hopf surface as described in~\cite[Theorem 2]{MR1477631}. Furthermore, if $s^W$ is nowhere zero then $\chi=\sigma=0,$ where $\chi$ and $\sigma$ are the Euler class and signature of $M$ respectively.
        \end{enumerate}
    \end{thm}
    \begin{proof}
        From Lemma~\ref{killing}, $\theta^\sharp$ is a Killing vector field. Since $d\theta=0$, it follows that $\theta$ is $D^g$-parallel and so $\theta^\sharp$ has a constant length. Combining Corollary~\ref{ric_star_lee} and Corollary~\ref{ric-expression} and (\ref{diff_ricci+sca}), we obtain
        \begin{equation}
            \left(s^H+\|N\|^2-2\|\theta\|^2\right)\|\theta\|^2=0.
        \end{equation}
        Hence either $\theta=0$ or $s^H=2\|\theta\|^2-\|N\|^2.$ Now, if $s^H=2\|\theta\|^2-\|N\|^2$ then $s^g=\frac{3}{2}\|\theta\|^2-2\|N\|^2$. Hence, from~(\ref{conformal_sca}), we get that
        $$s^W=-2\|N\|^2.$$
        Thus $s^W\equiv 0$ if and only if $J$ is integrable. Moreover, if $s^W$ is nowhere zero then $5\chi+6\sigma=0$ using~\cite[Lemma 3]{MR1604803}. The existence of a non-trivial Killing vector field of constant length implies $\chi=0$ by Hopf theorem~\cite{MR1512316} hence $\chi=\sigma=0$. Furthermore, using~(\ref{diff_sca}) we have that $s^\star=s^g-\|\theta\|^2+2\|N\|^2=\frac{3}{2}\|\theta\|^2-2\|N\|^2-\|\theta\|^2+2\|N\|^2=\frac{1}{2}\|\theta\|^2$ and so $s^\star$ is constant.
    \end{proof}

    \begin{rem}
        We recall that there are many restrictions to the existence of a non-zero Killing vector field of constant length on a Riemannian manifold, see for example~\cite{MR2442533,MR3422911}
    \end{rem}

    \begin{rem}
        When $J$ is integrable, if $\theta$ is $D^g$-parallel (the metric $g$ is called Vaisman) then $\theta^\sharp$ and $J\theta^\sharp$ are both Killing and real holomorphic vector fields~\cite{MR1481969}.  
    \end{rem}

    Gauduchon proved in~\cite{MR1363825} (in dimension $3$ see also~\cite{MR1200290}) that if the conformal scalar curvature of a compact Einstein--Weyl manifold is non-positive but non identically zero then the Gauduchon metric is an almost-K\"ahler Riemannian Einstein metric (i.e. $Ric^g$ is proportional to $g$). We can then deduce the following 
    \begin{cor}~\cite[Theorem 3]{MR1363825}\cite[Corollary 4.2]{MR1727302}
        Let $(M,\tilde{g},J)$ be a $4$-dimensional compact locally conformally symplectic manifold. Suppose that the Gauduchon metric $g\in[\tilde{g}]$ is a second-Chern--Einstein and an Einstein--Weyl metric. Then $(M,g,J)$ is either an almost-K\"ahler Riemannian Einstein manifold or a Hopf surface as described in~\cite{MR1477631}.
    \end{cor}
    \begin{proof}
        It follows from Theorem~\ref{thm1} that either $g$ is almost-K\"ahler or $s^W$ is non-positive. The corollary follows from the fact that if the metric $g$ is an almost-K\"ahler Einstein--Weyl metric then $g$ is a Riemannian Einstein metric~\cite{MR1363825}.
    \end{proof}

    The vanishing of $N_{\theta^\sharp}$ means that $\theta^\sharp$ is $g$-orthogonal to $Span(N)$, which is the distribution spanned by all the vector fields $N(X,Y).$ Moreover, if $N_{\theta^\sharp}=0$ then $N_{J\theta^\sharp}=0$. In fact, in real dimension $4$, at each point the dimension of $Span(N)$ is equal to $0$ or $2$~\cite{Cahen:2020vv} (see also~\cite{MR840830} for more details). A similar proof to Theorem~\ref{thm1} gives the following  

    \begin{thm}\label{N-theta}
        Suppose that $(M,\tilde{g},J)$ is a $4$-dimensional compact almost-Hermitian manifold and $\tilde{g}$ is a second-Chern--Einstein metric. Suppose that the Gauduchon metric $g$ in the conformal class $[\tilde{g}]$ satisfies $(D^g\theta)^{sym,J,-}=0$ and $N_{\theta^\sharp}=0$, where $\theta$ is the Lee form of $g$. Then, either \begin{enumerate} 
            \item $(M,g,J)$ is a second-Chern--Einstein almost K\"ahler manifold, or\\
            \item $\theta^\sharp$ is a non-zero Killing vector field and the conformal scalar curvature $s^W$ is non-positive and the $\star$-scalar curvature $s^\star$ is positive. Moreover, $s^W$ is identically zero if and only if $J$ is integrable so $(M,J)$ is a Hopf surface as described in~\cite{MR1477631}. In addition, if $s^W$ is nowhere zero then $\chi=\sigma=0.$ Furthermore, $s^\star$ is a (positive) constant if and only if $\theta^\sharp$ is a non-zero Killing vector field of constant length.
        \end{enumerate}
    \end{thm}

    Now, we would like to investigate the condition $N_{\theta^\sharp}=0$ and see if it can be implied by the $J$-invariance of different Ricci forms.
    \begin{lem}\label{ricci-hermitian}
        Let $(M,J,g)$ be a $4$-dimensional almost-Hermitian manifold. Then
        $$\rho^\nabla=R^W(F)-dJ\theta-\left(d\theta\right)^{(2,0)+(0,2)}_{J\cdot,\cdot}-\frac{1}{2}\sum_{i=1}^4\left(N(e_i,e_j)\right)^\flat\wedge\left(JN(e_i,e_j)\right)^\flat,$$
        where $\flat$ is the isomorphism between vector fields and $1$-forms induced by $g.$ In particular $$\left(\rho^\nabla\right)^{(2,0)+(0,2)}=\left(R^W(F)\right)^{(2,0)+(0,2)}-\left(dJ\theta\right)^{(2,0)+(0,2)}-\left(d\theta\right)^{(2,0)+(0,2)}_{J\cdot,\cdot}.$$
    \end{lem}
    \begin{proof}
        From Proposition~\ref{relation-chern-weyl}, we have that
        \begin{align*}
            2\rho^\nabla(X,Y)=&\sum_{i=1}^4g(R^\nabla_{X,Y}e_i,Je_i),\\
            =&\sum_{i=1}^4g(R^W_{X,Y}e_i,Je_i)-\frac{1}{2}\left(dJ\theta\right)(X,Y)g(Je_i,Je_i)-\frac{1}{2}\left(d\theta\right)(X,Y)g(e_i,Je_i)\\
            &-\frac{1}{2}g\left(\left(D^W_X\left(D^W_YJ\right)-D^W_Y\left(D^W_XJ\right)-D^W_{[X,Y]}J \right)Je_i,Je_i\right)\\
            &+\frac{1}{4}g\left(\left(\left(D^W_XJ\right)\left(D^W_YJ\right) -\left(D^W_YJ\right)\left(D^W_XJ\right)   \right)e_i,Je_i\right),\\
            =&\sum_{i=1}^4g(R^W_{X,Y}e_i,Je_i)-2\left(dJ\theta\right)(X,Y)\\
            &+\frac{1}{4}\sum_{i=1}^4g\left(\left(\left(D^W_XJ\right)\left(D^W_YJ\right) -\left(D^W_YJ\right)\left(D^W_XJ\right)   \right)e_i,Je_i\right).
        \end{align*}
        Because $D^W$ is torsion free, and using the relation
        $$g(R^W_{X,Y}Z,W)+g(R^W_{X,Y}W,Z)=(d\theta)(X,Y)g(Z,W),$$
        we can easily deduce  the following relation for any vector fields $X,Y,Z,W$
        \begin{multline*}
            2g(R^W_{X,Y}Z,W)=\,2g(R^W_{Z,W}X,Y)-d\theta(Z,W)g(X,Y)+d\theta(X,Y)g(Z,W) -d\theta(X,W)g(Y,Z) \\ -d\theta(Y,Z)g(X,W)+d\theta(Y,W)g(X,Z)+d\theta(X,Z)g(Y,W).
        \end{multline*}
        In particular,
        $$\displaystyle\sum_{i=1}^4g(R^W_{e_i,Je_i}X,Y)=\sum_{i=1}^4g(R^W_{X,Y}e_i,Je_i)+d\theta(JX,Y)+d\theta(X,JY).$$
        Hence,
        \begin{multline}\label{hermitian-ricci}
            2\rho^\nabla(X,Y)=\,2R^W(F)(X,Y)-d\theta(JX,Y)-d\theta(X,JY)-2\left(dJ\theta\right)(X,Y)\\
            +\frac{1}{4}\sum_{i=1}^4g\left(\left(\left(D^W_XJ\right)\left(D^W_YJ\right) -\left(D^W_YJ\right)\left(D^W_XJ\right)   \right)e_i,Je_i\right).
        \end{multline}
        Moreover, using~(\ref{int_weyl}) we have
        \begin{align*}
            \sum_{i=1}^4g\left(\left(D^W_XJ\right)\left(D^W_YJ\right)e_i,Je_i\right)&=-2\sum_{i=1}^4g\left(N\left(\left(D^W_YJ\right)e_i,e_i\right),X\right),\\
            &=-2\sum_{i,j=1}^4g\left(N\left(e_j,e_i\right),X\right)g\left( \left(D^W_YJ\right)e_i,e_j \right),\\
            &=-4\sum_{i,j=1}^4g\left(N\left(e_i,e_j\right),X\right)g\left(JN\left(e_i,e_j\right),Y\right).
        \end{align*}
        Hence
        \begin{align*}
            \sum_{i=1}^4g\left(\left(\left(D^W_XJ\right)\left(D^W_YJ\right) -\left(D^W_YJ\right)\left(D^W_XJ\right)   \right)e_i,Je_i\right)=&
            -4\sum_{i,j=1}^4g\left(N\left(e_i,e_j\right),X\right)g\left(JN\left(e_i,e_j\right),Y\right)\\
            &+4\sum_{i,j=1}^4g\left(N\left(e_i,e_j\right),Y\right)g\left(JN\left(e_i,e_j\right),X\right),\\
            =&-4\sum_{i,j=1}^4\left(\left(N\left(e_i,e_j\right)\right)^\flat \wedge\left(JN\left(e_i,e_j\right)\right)^\flat \right)(X,Y).
        \end{align*}
        From~(\ref{hermitian-ricci}), we deduce that
        \begin{equation*}
        \begin{split}
            \rho^\nabla(X,Y)=\,R^W(F)(X,Y)-\frac{1}{2}\left(d\theta(JX,Y)+d\theta(X,JY)\right)-\left(dJ\theta\right)(X,Y) \\
            -\frac{1}{2}\sum_{i,j=1}^4\left(N\left(e_i,e_j\right)\right)^\flat \wedge\left(JN\left(e_i,e_j\right)\right)^\flat(X,Y).
        \end{split}
        \end{equation*}
    \end{proof}

    \begin{rem}
        We can compute the $J$-anti-invariant part of $R^W(F)$ and it is given by
        \begin{equation*}
            \left(R^W(F)\right)^{(2,0)+(0,2)}=-\sum_{i=1}^4g\left(\left(D^g_{e_i}N \right)(\cdot,\cdot),e_i\right)+\frac{3}{2}N_{\theta^\sharp},
        \end{equation*}
        where $N_{\theta^\sharp}=g(N(\cdot,\cdot),\theta^\sharp).$
    \end{rem}

    \begin{cor}\label{Ntheta}
        Let $(M,g,J)$ be a $4$-dimensional almost-Hermitian manifold. Suppose that $\rho^\nabla$ and $R^W(F)$ are $J$-invariant. Then $N_{\theta^\sharp}=0$.
    \end{cor}
    \begin{proof}
        From Lemma~\ref{ricci-hermitian}, we obtain that 
        \begin{equation}\label{derive-theta}
            \left(dJ\theta\right)^{(2,0)+(0,2)}=-\left(d\theta\right)^{(2,0)+(0,2)}_{J\cdot,\cdot}.
        \end{equation}
        Applying the Lie derivative $\mathcal{L}_{\theta^\sharp}$ to the relation $F=g(J\cdot,\cdot)$ and using the Cartan formula we obtain
        $$dJ\theta=-\|\theta\|^2F+\theta\wedge J\theta+2\left(D^g\theta\right)^{sym}_{J\cdot,\cdot}+g\left(\mathcal{L}_{\theta^\sharp}J\cdot,\cdot\right).$$
        In particular
        \begin{equation}\label{dJtheta} 
            \left(dJ\theta\right)^{(2,0)+(0,2)}=\left(\mathcal{L}_{\theta^\sharp}J\right)^{anti-sym}.
        \end{equation}
        Similarly, we have
        $$d\theta=-2\left(D^gJ\theta \right)^{sym}_{J\cdot,\cdot}-g\left(\mathcal{L}_{J\theta^\sharp}J\cdot,\cdot\right).$$ In particular,
        \begin{equation}\label{dTheta}
            \left(d\theta\right)^{(2,0)+(0,2)}_{J\cdot,\cdot}=g\left(J\left(\mathcal{L}_{J\theta^\sharp}J\right)^{anti-sym}\cdot,\cdot\right).
        \end{equation}
        Combining~(\ref{derive-theta})~(\ref{dJtheta}) and~(\ref{dTheta}) we deduce that 
        \begin{equation}\label{final}
            \left(\mathcal{L}_{J\theta^\sharp}J\right)^{anti-sym}=J\left(\mathcal{L}_{\theta^\sharp}J\right)^{anti-sym}.
        \end{equation}
        Now, for any almost-Hermitian manifold we have $$\mathcal{L}_{J\theta^\sharp}J-J\left(\mathcal{L}_{\theta^\sharp}J\right)=4N(\theta^\sharp,\cdot).$$ On the other hand from~(\ref{cycle_n}) we see that $\left(g\left((N(\theta^\sharp,\cdot),\cdot\right)\right)^{anti-sym}=-{2}N_{\theta^\sharp}$ so that
        $$g\left(\left(\mathcal{L}_{J\theta^\sharp}J\right)^{anti-sym}-J\left(\mathcal{L}_{\theta^\sharp}J\right)^{anti-sym}\cdot,\cdot\right)=-{2}N_{\theta^\sharp}.$$
        From~(\ref{final}), we deduce that $N_{\theta^\sharp}=0.$
    \end{proof}

    \begin{cor}
        Suppose that $(M,\tilde{g},J)$ is a compact $4$-dimensional almost-Hermitian manifold and $\tilde{g}$ is a second-Chern--Einstein metric. Suppose that the Gauduchon metric $g$ in the conformal class $[\tilde{g}]$ satisfies $(D^g\theta)^{sym,J,-}=0$ and that $\rho^\nabla$ and $R^W(F)$ are $J$-invariant. Then conclusions of Theorem~\ref{N-theta} hold.
    \end{cor}
    \begin{proof}
        $\rho^\nabla$ and $R^W(F)$ being $J$-invariant implies that $N_{\theta^\sharp}=0$ by Corollary~\ref{Ntheta}. We apply then Theorem~\ref{N-theta}.
    \end{proof}
    
\subsection{Relation between second Chern--Ricci and Bismut--Ricci tensors}
    Finally, we relate the second Chern--Ricci tensor with the Bismut--Ricci tensor. This will give us the possibility to compute it on almost-abelian Lie groups in Section~\ref{Sec_4}. Thanks to Corollary~\ref{cor_2-chern} and Lemma~\ref{ricci-hermitian} we see that both the first Chern--Ricci form $\rho^\nabla$ and the second Chern--Ricci form $r$ can be expressed in terms of $R^W(F)$. Hence, we get the following
    \begin{prop}\label{prop: second chern first bismut}
        Let $(M,g,J)$ be a $4$-dimensional almost-Hermitian manifold. Then,
        \begin{equation*}
            r = \left(\rho^\nabla + dJ\theta \right)^{J,+} + \frac{1}{4}\left(2\delta^g\theta + 2\|\theta\|^2 - \|N\|^2 \right)F + \frac{1}{2}\sum_{i,j}^4\left(N(e_i,e_j)\right)^{\flat}\wedge\left(J N(e_i,e_j)\right)^{\flat}.
        \end{equation*}
        In particular, thanks to ~(\ref{eq: ricci chern and bismut}),
        \begin{equation*}
            r = \left(Ric^B\right)^{J,+} + \frac{1}{4}\left(2\delta^g\theta + 2\|\theta\|^2 - \|N\|^2 \right)F + \frac{1}{2}\sum_{i,j}^4\left(N(e_i,e_j)\right)^{\flat}\wedge\left(J N(e_i,e_j)\right)^{\flat}\;.
        \end{equation*}
    \end{prop}
    From this proposition we see that in the integrable case the second-Chern--Einstein problem agrees with the Bismut--Einstein problem 
    $$ \left(Ric^B\right)^{J,+} = \lambda \, F, \quad \text{for some function } \lambda. $$
    As a matter of fact, on a $4$-dimensional Hermitian manifold 
    $$ R^W(F) = \left(Ric^B\right)^{J,+}.\; $$
    This also follows from Lemma~\ref{ricci-hermitian} and ~(\ref{eq: ricci chern and bismut}). 
    \begin{rem}
        One can check that the crucial property of $4$-dimensional manifolds that leads to these relations between the second Chern--Ricci form, the first Bismut--Ricci form and $R^W(F)$ is $dF=\theta\wedge F$. As a consequence, in higher dimension $2n$, if we assume the Hermitian structure to be locally conformally K\"ahler we obtain similar relations (see ~\cite{MR1836272}). For example, computations analogous to the one we did here can show that on a LCK manifold
        $$ R^W(F) = Ric^{(\frac{1}{1-n})}\;, $$
        where $Ric^{(t)}$ is the first Ricci form of the Gauduchon connection $\nabla^t$ introduced in~\cite{MR1456265}, see Section~\ref{Sec_4} for the definition.
    \end{rem}

\section{$4$-dimensional compact examples with almost-Hermitian second-Chern--Einstein metrics}\label{Sec_3}
Here we use the same notation of Lie algebras as~\cite{MR404362}.
\subsection{Conformally locally symplectic Lie algebras associated to compact solvmanifolds}

\begin{enumerate}
\item{The Lie algebra $\mathcal{A}_{3,6}\oplus \mathcal{A}_1$}: the structure of the Lie algebra is 
$$[e_1,e_3]=-e_2,\quad [e_2,e_3]=e_1,$$
where $\{e_1,e_2,e_3,e_4\}$ is a basis of $\mathcal{A}_{3,6}\oplus \mathcal{A}_1$.
The associated simply connected group to the Lie algebra $\mathcal{A}_{3,6}\oplus \mathcal{A}_1$ admits lattices 
(see for example~\cite{MR3480018,MR3763412,MR4088745}, in the notation of~\cite{MR3763412} $\mathcal{A}_{3,6}\oplus \mathcal{A}_1$ corresponds to $\mathfrak{r}^\prime_{3,0}\times\mathbb{R}$).
We consider the almost-complex structure
$$Je_1=e_3,\quad Je_2=e_4.$$
The almost-complex structure $J$ is non-integrable because $N(e_1,e_2)=\frac{1}{4}e_3.$

We consider the following $J$-compatible metric $g$
$$g=\left(\sqrt{5}-1\right)\left(e^1\otimes e^1+e^3\otimes e^3\right)+e^2\otimes e^2+e^4\otimes e^4,$$
where $\{e^1,e^2,e^3,e^4\}$ is the dual basis.

The pair $(g,J)$ induces the fundamental form
$$F=\left(\sqrt{5}-1\right)\,e^{13}+e^{24},$$
where $e^{13}=e^{1}\wedge e^3$ etc. Remark that the form $dF=e^{134}.$ Moreover, the Lee form is given by
$$\theta=\frac{1}{\left(\sqrt{5}-1\right)}e^4.$$
Hence $d\theta=0.$ Moreover, $(D^g\theta)^{sym,J,-}=0.$ On the other hand, the second Chern--Ricci form is given by $$r=\frac{1}{4}\,e^{13}+\frac{1}{4\left(\sqrt{5}-1\right)}e^{24},$$
so the metric $g$ is a second-Chern--Einstein metric with a positive Hermitian scalar curvature $s^H=\frac{1}{\sqrt{5}-1}$.
Thus, $\theta$ is $D^g$-parallel. We also remark that $N_{\theta^\sharp}=0$ and the first Chern--Ricci form $\rho^\nabla=\frac{1}{2}\,e^{13}$ is J-invariant.\\

\item{The Lie algebra $\mathcal{A}_{4,1}$}: the structure of the Lie algebra is 
$$[e_2,e_4]=e_1,\quad [e_3,e_4]=e_2,$$
where $\{e_1,e_2,e_3,e_4\}$ is a basis of $\mathcal{A}_{4,1}$.
The associated simply connected group to the Lie algebra $\mathcal{A}_{4,1}$ admits lattices 
(see for example~\cite{MR3480018,MR3763412,MR4088745}, in the notation of~\cite{MR3763412} $\mathcal{A}_{4,1}$ corresponds to $\mathfrak{n}_{4}$).

We consider the almost-complex structure
$$Je_1=e_3,\quad Je_2=e_4.$$
The almost-complex structure $J$ is non-integrable because $N(e_1,e_2)=\frac{1}{4}e_2.$

We consider the following $J$-compatible metric $g$
$$g=\frac{1}{2}\left(e^1\otimes e^1+e^3\otimes e^3\right)+e^2\otimes e^2+e^4\otimes e^4,$$
where $\{e^1,e^2,e^3,e^4\}$ is the dual basis.

The pair $(g,J)$ induces the fundamental form
$$F=\frac{1}{2}\,e^{13}+e^{24}.$$
Remark that the form $dF=\frac{1}{2}\,e^{234}.$ Moreover, the Lee form is given by
$$\theta=-\frac{1}{2}e^3.$$
Hence $d\theta=0.$ However in this example, $(D^g\theta)^{sym,J,-}$ does not vanish so $\theta^\sharp$ is not a Killing vector field. Explicitly, $(D^g\theta)^{sym,J,-}=\frac{1}{2}\left(e^2\otimes e^4 +e^4\otimes e^2\right)$. On the other hand, the second Chern--Ricci form
$r$ vanishes,
so the metric $g$ is a second-Chern--Einstein metric with vanishing Hermitian scalar curvature.
We also remark that $N_{\theta^\sharp}=0$ and the first Chern--Ricci form vanishes.
\\

\item{The Lie algebra $\mathcal{A}_{4,8}$}: the structure of the Lie algebra is 
$$[e_2,e_3]=e_1,\quad [e_2,e_4]=e_2,\quad [e_3,e_4]=-e_3,$$
where $\{e_1,e_2,e_3,e_4\}$ is a basis of $\mathcal{A}_{4,8}$.
The associated simply connected group to the Lie algebra $\mathcal{A}_{4,8}$ admits lattices 
(see for example~\cite{MR3480018,MR4088745}, in the notation of~\cite{MR4088745} $\mathcal{A}_{4,8}$ corresponds to $\mathfrak{d}_{4}$)
We consider the almost-complex structure
$$Je_1=e_4,\quad Je_2=e_3.$$
The almost-complex structure $J$ is non-integrable because $N(e_1,e_2)=\frac{1}{2}e_3.$

We consider the following $J$-compatible metric $g$
$$g=\sum_{i=1}^4e^i\otimes e^i,$$
where $\{e^1,e^2,e^3,e^4\}$ is the dual basis.

The pair $(g,J)$ induces the fundamental form
$$F=e^{14}+e^{23}.$$
Remark that the form $dF=-e^{234}.$ Moreover, the Lee form is given by
$$\theta=-e^4.$$
Hence $d\theta=0.$ However in this example, $(D^g\theta)^{sym,J,-}$ does not vanish so $\theta^\sharp$ is not a Killing vector field. Explicitly, $(D^g\theta)^{sym,J,-}=e^3\otimes e^3 -e^2\otimes e^2.$ On the other hand, the second Chern--Ricci form
$r$ vanishes,
so the metric $g$ is a second-Chern--Einstein metric with vanishing Hermitian scalar curvature.
We also remark that $N_{\theta^\sharp}=0$ and the first Chern--Ricci form vanishes.\\

\end{enumerate}

\subsection{Non conformally locally symplectic Lie algebra associated to compact solvmanifolds}


{The Lie algebra $\mathcal{A}_{4,10}$}: the structure of the Lie algebra is 
$$[e_2,e_3]=e_1,\quad [e_2,e_4]=-e_3,\quad [e_3,e_4]=e_2,$$
where $\{e_1,e_2,e_3,e_4\}$ is a basis of $\mathcal{A}_{4,10}$.
The associated simply connected group to the Lie algebra $\mathcal{A}_{4,8}$ admits lattices 
(see for example~\cite{MR3480018,MR4088745}, in the notation of~\cite{MR4088745} $\mathcal{A}_{4,10}$ corresponds to $\mathfrak{d}^\prime_{4,0}$).
We consider the almost-complex structure
$$Je_1=e_3,\quad Je_2=e_4.$$
The almost-complex structure $J$ is non-integrable because $N(e_1,e_2)=\frac{1}{4}e_2+\frac{1}{4}e_3.$

We consider the following $J$-compatible metric $g$
$$g=\frac{1+\sqrt{17}}{8}(e^1\otimes e^1+e^3\otimes e^3)+(e^2\otimes e^2+e^4\otimes e^4),$$
where $\{e^1,e^2,e^3,e^4\}$ is the dual basis.

The pair $(g,J)$ induces the fundamental form
$$F=\frac{1+\sqrt{17}}{8}\,e^{13}+e^{24}.$$
Remark that the form $dF=-\frac{1+\sqrt{17}}{8}\,e^{124}.$ Moreover, the Lee form is given by
$$\theta=-\frac{1+\sqrt{17}}{8}\,e^1.$$
Hence $d\theta\neq 0.$ In this example, $(D^g\theta)^{sym,J,-}$ vanishes. On the other hand, the second Chern--Ricci form is
$$r=\frac{1+\sqrt{17}}{32}\,e^{13}+\frac{1}{4}\,e^{24}$$
so the metric $g$ is a second-Chern--Einstein metric with positive Hermitian scalar curvature $s^H=1.$ Hence, $\theta^\sharp$ is a Killing vector field but not $D^g$-parallel. 
We also remark that in this example $N_{\theta^\sharp}\neq 0$ and the first Chern--Ricci form $\rho^\nabla=\frac{1}{2}\,e^{24}-\frac{1}{2}\,e^{34}$ is not $J$-invariant.\\

\section{Special metrics metrics on almost-abelian Lie algebras}\label{Sec_4}

 An {almost-abelian} Lie group $G$ is a Lie group whose Lie algebra $\g$ has a codimension-one abelian ideal $\n\subset \g$. Given an almost-Hermitian left-invariant structure $(g,J)$ on a $2n$-dimensional almost-abelian Lie group $G$, define $\n_1:=\n\cap J\n$ and $J_1: = J_{|_{\n_1}}$. Then we can choose an orthonormal basis $\{e_1,\ldots,e_{2n}\}$ for $\g$ such that 

   $$ \n=\mbox{span}_\R\left<e_1,\ldots,e_{2n-1}\right> \quad \text{ and } \quad Je_i=e_{2n-i+1} \text{ for } i=1,\ldots,n.  $$

    Hence, the fundamental form $F(\cdot,\cdot):=g(J\cdot,\cdot)$ associated to the almost-Hermitian structure $(J,g)$ is
    $$ F = e^1\wedge e^{2n} + e^2\wedge e^{2n-1} + \cdots + e^n\wedge e^{n+1},$$
    given in terms of the dual left-invariant frame $\{e^1,\ldots,e^{2n}\}$.

    The algebra structure of $\g$ is completely described by the adjoint map 
    \begin{eqnarray*}  
    \mbox{ad}_{e_{2n}}: \g& \rightarrow& \g \\
     x &\mapsto& [e_{2n},x].\end{eqnarray*}
     The matrix associated to this endomorphism is
    
    \begin{equation}\label{eq: ad_e_2n}
        \ad_{e_{2n}|_\n} = \begin{pmatrix} 
        a & b \\
        v & A
        \end{pmatrix}, \quad a \in\R,\; b,v \in \n_1,\; A\in\mathfrak{gl}(\n_1).
    \end{equation}
    
    The data $(a,b,v,A)$ completely characterizes the almost-Hermitian structure $(g,J)$. For example, the integrability of $J$ can be expressed in terms of $(a,b,v,A)$ asking that $b=0$ and $A\in\mathfrak{gl}(n_1,J_1)$, where $\mathfrak{gl}(n_1,J_1)$ denotes endomorphisms of $\n_1$ commuting with $J_1$, see \cite[Lemma 4.1]{MR3957836}.
    
    On an almost-abelian almost-Hermitian Lie group $\left(G,[\cdot,\cdot]_{(a,b,v,A)},J,g\right)$ the Lee form is given by
    \begin{equation}\label{lee-form}
        \theta = J\delta^gF = (Jv)^{\flat} - (\tr\,A)e^{2n},
    \end{equation}
    with respect to the adapted unitary basis $\{e_1,\ldots,e_{2n}\}$, see for example~\cite{Fino:2020tr}.
  
    Given an almost-Hermitian manifold $(M,g,J)$, Gauduchon introduced in~\cite{MR1456265} a $1$-parameter family $\Na^{(t)}$ of canonical Hermitian connections, i.e. $\Na^{(t)} g = \Na^{(t)} J = 0$. 
    In this family $\Na^{(1)}=\nabla$ corresponds to the Chern connection; while $\Na^{(-1)}=\nabla^B$ corresponds to the Bismut connection.
%
    Any canonical connection $\Na^{(t)}$ induces the associated {first Ricci form} $$Ric^{(t)}= \frac{1}{2}\sum_{i=1}^{2n}g(R^{\nabla^t}_{X,Y}e_i,Je_i),$$ 
    where $R^{\nabla^{(t)}}$ denotes the curvature tensor of $\Na^{(t)}$. 
    
    The first Ricci forms of the canonical connections on a Lie group $(G,\g)$ equipped with an almost-Hermitian structure $(g,J)$ were computed in~\cite{MR2988734}. In particular, for any parameter $t\in\R$ these are given by
    
    $$ Ric^{(t)}(X,Y) = - \frac{1}{2}\left\{ \tr\left(\ad_{[X,Y]}\circ J\right) - t \, \tr\,\ad_{J[X,Y]} + (t-1)g\left(F,d [X,Y]^{\flat}\right)  \right\}\;. $$

    Then, a direct computation leads to the following Lemma
    \begin{lem}
        Let $\left(G,[\cdot,\cdot]_{(a,b,v,A)},g,J\right)$ be an almost-abelian almost-Hermitian Lie group, endowed with an adapted unitary basis $\{e_1,\ldots,e_{2n}\}$, determining the algebraic data $(a,b,v,A)$ by (\ref{eq: ad_e_2n}). Then, the first Ricci form associated to the canonical connection $\Na^t$ is 
        \begin{equation*}
            Ric^{(t)} = -\frac{1}{2}\left\{\left( 2a^2 + t\,a\,\tr\,A + (1-t)\left|v\right|^2 + b\cdot v\right) e^1 + \left( (2a + t\,\tr\,A)b + A^t b + (1-t)A^t v\right)^{\flat}\right\}\wedge e^{2n}. 
        \end{equation*}
        In particular,
        \begin{equation}\label{eq: first Bismut Ricci form A-A}
            Ric^{B} = -\frac{1}{2}\left\{\left( 2a^2 - a\,\tr\,A + 2\left|v\right|^2 + b\cdot v\right) e^1\wedge e^{2n} + \left( (2a - \tr\,A)b + A^t b + 2A^t v\right)^{\flat}\wedge e^{2n}\right\}. 
        \end{equation}
    \end{lem}
    
    \medskip
    
    Before studying the second-Chern--Einstein problem, we examine the Bismut--Einstein problem on almost-Hermitian almost-abelian Lie groups of any dimension. We analyze the Einstein condition together with the request that $dJdF=0$. Indeed, in Hermitian geometry the Einstein problem for the Bismut connection is stated as follows
    \begin{defn}[\cite{Gar-Str}]
        Let $(M,g,J)$ be a Hermitian manifold with $g$ a pluriclosed metric (meaning that $dJdF=0$). Then $g$ is said to be a {\em Bismut-Hermitian-Einstein metric} if $\left(Ric^B\right)^{J,+} = \lambda\,F$ for some function $\lambda$.
    \end{defn}
    These metrics were firstly studied in~\cite{MR2673720} as fixed points of a parabolic flow of metrics, the {\em pluriclosed flow}. We shall remark that there is a lack of examples of such metrics: the only known are the K\"ahelr--Einstein and the Bismut-flat metrics, the latter meaning that the whole Bismut curvature tensor vanishes ($R^{B}=0$). In the following, we prove that on  Hermitian almost-abelian Lie groups the Bismut-Hermitian-Einstein metrics are K\"ahler. However, as soon as we drop the the integrability assumption we are able to find non-almost-K\"ahler metrics that satisfy the Bismut--Einstein equation and such that $dJdF=0$. 

\subsection{Bismut--Einstein Problem}
    We study here the Bismut--Einstein problem, $\left(Ric^B\right)^{J,+}=\lambda\, F$ for $\lambda\in\R$, on an almost-Hermitian almost-abelian Lie group $\left(G,[\cdot,\cdot]_{(a,b,v,A)},g,J\right)$. Thanks to Equation (\ref{eq: first Bismut Ricci form A-A}) we see that to obtain a solution to the Einstein problem $\lambda$ must vanish. Thus, we obtain the conditions
    \begin{align}\label{eq: Bismut Ricci flat A-A}
        \begin{cases}
            2a^2 - a\,\tr\,A + 2 |v|^2 + b\cdot v = 0, &\\
           \left(2a - \tr\,A\right)b + A^t b + 2A^t v= 0, &
        \end{cases}
    \end{align} 
    
    Moreover, from a direct computation of $dJdF$ we obtain the following lemma
    
    \begin{lem}
        Let $\left(G,[\cdot,\cdot]_{(a,b,v,A)},g,J\right)$ be an almost-Hermitian almost-abelian Lie group, endowed with an adapted unitary basis $\{e_1,\ldots,e_{2n}\}$, determining the algebraic data $(a,b,v,A)$ by (\ref{eq: ad_e_2n}). Then, the metric satisfies $dJdF=0$ (here $JdF=-dF(J\cdot,J\cdot,J\cdot)$) if and only if for any $x,y,z \in \n_1$
        \begin{equation*}
            \begin{cases}
                    \left<b,x\right>\left(\left<AJy,z\right> - \left<AJz,y\right>\right) - 
                    \left<b,y\right>\left(\left<AJx,z\right> - \left<AJz,x\right>\right) + 
                    \left<b,z\right>\left(\left<AJx,y\right> - \left<AJy,x\right>\right) = 0,\\
                    a\left(\left<AJy,z\right> - \left<AJz,y\right>\right) + \left<AJAy,z\right> - \left<AJz,Ay\right> + \left<AJy,Az\right> - \left<AJAz,y\right> = 0.
            \end{cases}
        \end{equation*}
        
    \end{lem}
    \begin{rem}
    In the Hermitian case, i.e. when $b=0$ and $A$ commutes with $J$, the condition $dJdF=0$ (the metric $g$ is then called SKT~\cite{MR1006380}) is equivalent to ask that
        \begin{equation}\label{eq: SKT condition on A-A}
            aA + A^2 + A^tA \in \mathfrak{so}(\n_1) \;,
        \end{equation}
        as showed in \cite[Lemma 4.4]{MR3957836}.
    \end{rem}
    Thus, in the Hermitian case, taking the trace in (\ref{eq: SKT condition on A-A}) we see that $a\,\tr\,A \leq 0$ with equality if and only if $A\in\mathfrak{so}(\n_1)$. This, together with the first equation in (\ref{eq: Bismut Ricci flat A-A}) (with $b=0$), imply that $v=0$ and $A \in\mathfrak{so}(\n_1)$, hence also $tr\,A = 0$ and then $\theta = 0$ (see Equation (\ref{lee-form})). We have the following 
    \begin{prop}
        On a Hermitian almost-Abelian Lie group $\left(G,[\cdot,\cdot]_{(a,b,v,A)},g,J\right)$ the Bismut-Hermitian-Einstein metrics, are K{\"a}hler-Einstein.
    \end{prop}
    On the other hand, if we take $A\in\mathfrak{so}(\n_1)$ commuting with $J$ but $b\neq 0 \neq v$, we obtain an almost-Hermitian almost-abelian Lie group $\left(G,[\cdot,\cdot]_{(a,b,v,A)},g,J\right)$ which satisfies $dJdF=0$ but is not almost-K{\"a}hler (i.e. $\theta\neq 0$). Then, any solution of the equations 
    \begin{align}
        \begin{cases}
            2a^2 + 2 |v|^2 + b\cdot v = 0 \\
            2ab + A^tb + 2A^tv = 0 
        \end{cases}
    \end{align} 
    gives a first-Bismut--Einstein metric on $G$. We focus on the $4$-dimensional unimodular case ($a=-tr\,A$), which corresponds to $a=0$.
     
    \begin{thm}\label{thm: classification-Bismut}
        Let $\g$ be a $4$-dimensional unimodular almost-abelian Lie algebra equipped with a left-invariant almost-Hermitian structure $(g,J)$ such that $A\in\mathfrak{so}(\n_1)$ and commutes with $J$, in particular it satisfies $dJdF=0.$ Suppose that $(g,J)$ is a solution to the Bismut--Einstein problem. Then $\g$ is isomorphic to one of the following Lie algebras
        \begin{enumerate}
            \item $\mathcal{A}_{3,6}\oplus \mathcal{A}_1:\,[e_1,e_3]=-e_2,\quad [e_2,e_3]=e_1$. The solution is not almost-K\"ahler. \\
            \item $\mathcal{A}_{3,1}\oplus \mathcal{A}_1:\,[e_1,e_2]=e_3$. The solution is almost-K\"ahler.
        \end{enumerate}
        Both Lie algebras admit compact quotients.
    \end{thm}
    \begin{proof}  
        The isomorphism classes of almost-abelian Lie algebras can be described using Jordan forms of $\ad_{e_{4}|_\n}$ up to scaling (see~\cite[Lemma 2.1]{MR3763412} and~\cite{MR2138348}).
        Denote by $A^i_j$ the $(i,j)$-th element of $A$. Equations~(\ref{eq: Bismut Ricci flat A-A}) become
        \begin{align}\label{eq: Bismut-4-dim}
            \begin{cases}
                2 |v|^2 + b\cdot v =& 0, \\
                A^1_2\,(b_1+2v_1)=& 0, \\
                A^1_2\,(b_2+2v_2)=&0.
            \end{cases}
        \end{align} 
        We have then two cases:
        \begin{enumerate}
            \item $A^1_2=0.$ Then, 
            \begin{itemize}
                \item either $b\cdot v=-2 |v|^2<0$: the canonical Jordan form of  $\ad_{e_{4}|_\n}$ up to scaling is
                $\begin{pmatrix} 
                    0 & 0 &0\\
                    0 & 0&1\\
                    0 & -1&0\\
                \end{pmatrix},$ which corresponds to $\mathcal{A}_{3,6}\oplus \mathcal{A}_1$. $v\neq 0$ so $dF\neq 0$.\\
                \item or $b\cdot v=v=0$: the canonical Jordan form of  $\ad_{e_{4}|_\n}$ up to scaling is
                $\begin{pmatrix} 
                    0 & 0 &0\\
                    0 &0 &1\\
                    0 & 0&0\\
                \end{pmatrix},$ which corresponds to $\mathcal{A}_{3,1}\oplus \mathcal{A}_1$. $v=\tr\,A=0$ so $dF= 0.$\\
            \end{itemize}
            \item $A^1_2\neq 0.$ Then, $b=-2v\neq 0$ and the canonical Jordan form of  $\ad_{e_{4}|_\n}$ up to scaling is
            $\begin{pmatrix} 
                0 & 0 &0\\
                0 & 0&1\\
                0 & -1&0\\
            \end{pmatrix},$ which corresponds to $\mathcal{A}_{3,6}\oplus \mathcal{A}_1$.  $v\neq 0$ so $dF\neq 0.$
        \end{enumerate}
   \end{proof}
    
    
\subsection{Second-Chern--Einstein Probelm}
    Here we study the second-Chern--Einstein problem on an almost-Hermitian almost-abelian Lie group $\left(G,[\cdot,\cdot]_{(a,b,v,A)},g,J\right)$ of real dimension 4. We have seen that in the integrable case it reduces to the the Bismut--Einstein problem~(Proposition~\ref{prop: second chern first bismut}), studied in the previous section; while in the non-integrable case a factor depending on the Nijenhuis tensor pops up:
    \begin{equation}\label{eq: Nijenhuis factor}
        \frac{1}{2}\sum_{i,j}^4\left(N_J(e_i,e_j)\right)^{\flat}\wedge\left(J N_J(e_i,e_j)\right)^{\flat}.
    \end{equation}
    Choose an adapted unitary basis $\{e_1,e_2,e_3,e_4\}$ for $\g$, determining the algebraic data $(a,b,v,A)$. Then (\ref{eq: Nijenhuis factor}) can be written as
    \begin{multline}\label{eq: Nijenhuis factor A-A}
        \frac{1}{2}\sum_{i,j}^4\left(N_J(e_i,e_j)\right)^{\flat}\wedge\left(J N_J(e_i,e_j)\right)^{\flat}= |b|^2 e^1\wedge e^4 + \left(\left(A^2_1 + A^1_2\right)^2 + \left(A^1_1 - A^2_2\right)^2\right)e^2\wedge e^3 \\
        + \left(b_2\left(A^1_1 - A^2_2\right) - b_1\left(A^1_2 + A^2_1\right)\right)e^1\wedge e^2 + \left(b_2\left(A^2_1 + A^1_2\right) + b_1\left(A^1_1 - A^2_2\right)\right)e^1\wedge e^3  \\ 
        + \left(b_1\left(A^1_1 - A^2_2\right) + b_2\left(A^1_2 + A^2_1\right)\right)e^2\wedge e^4 + \left(b_1\left(A^2_1 + A^1_2\right) - b_2\left(A^1_1 - A^2_2\right)\right)e^3\wedge e^4,
    \end{multline}
    where $b=(b_1,b_2)$ and $A^i_j$ is the $(i,j)$-th element of $A$.

\begin{thm}\label{thm: classification}
Let $\g$ be a $4$-dimensional unimodular almost-abelian Lie algebra equipped with a left-invariant almost-Hermitian non-Hermitian structure $(g,J)$ such that the Lee form $\theta$ is $D^g$-parallel and non-zero. Suppose
that $(g,J)$ is a solution to the second-Chern--Einstein problem. Then $\g$ is isomorphic to one of the following Lie algebras
\begin{enumerate}
\item $\mathcal{A}_{3,6}\oplus \mathcal{A}_1:\,[e_1,e_3]=-e_2,\quad [e_2,e_3]=e_1$.\\
\item $\mathcal{A}_{3,4}\oplus \mathcal{A}_1:\,[e_1,e_3]=e_1,\quad [e_2,e_3]=-e_2.$\\
\end{enumerate}
Both Lie algebras admit compact quotients.
\end{thm}
\begin{rem}
A classification of $4$-dimensional almost-abelian Lie algebras with locally conformally symplectic structure is given in~\cite{MR3763412} (see also~\cite{MR4088745}). In the notation of~\cite{MR3763412}, the above Lie algebras correspond respectively to $\mathfrak{r}^\prime_{3,0}\times\mathbb{R}$ and $\mathfrak{r}_{3,-1}\times\mathbb{R}.$
\end{rem}

 \begin{proof}   
We recall that the isomorphism classes of almost-abelian Lie algebras can be described using Jordan forms of $\ad_{e_{4}|_\n}$ up to scaling (see~\cite[Lemma 2.1]{MR3763412} and~\cite{MR2138348}). 

The Lie algebra $\g$ is unimodular i.e. $\tr\,A=-a.$ Thanks to equations (\ref{eq: first Bismut Ricci form A-A}) and Proposition~\ref{prop: second chern first bismut}, we get
     \begin{eqnarray}\label{eq: Chern-4-dim}
        \begin{cases}
            2|b|^2-3\,a^2- b\cdot v-2|v|^2 = 2(A^1_1-A^2_2)^2+2(A^1_2+A^2_1)^2, \\
           3ab_1+A^1_1b_1+A^2_1b_2+2A^1_1v_1+2A^2_1v_2= 4b_1(A^1_1-A^2_2)+4b_2(A^1_2+A^2_1), \\
           3ab_2+A^1_2b_1+A^2_2b_2+2A^1_2v_1+2A^2_2v_2= 4b_1(A^1_2+A^2_1)-4b_2(A^1_1-A^2_2).\nonumber
        \end{cases}
    \end{eqnarray} 

The Lee form is given by $$\theta=v_1\,e^3-v_2\,e^2+a\,e^4.$$ Then, the condition $D^g\theta=0$ implies the following
\begin{eqnarray*}
a&=&0,\\
b_1v_2&=&b_2v_1,\\
v_1(A^1_2-A^2_1)&=&0,\\
v_2(A^1_2-A^2_1)&=&0,\\
v_1(A^1_2+A^2_1)&=&2A^1_1v_2,\\
v_2(A^1_2+A^2_1)&=&2A^2_2v_1.
\end{eqnarray*}

Suppose that $v_1\neq 0.$ Then the above equations imply that $A^1_1=A^2_2=A^1_2=A^2_1=0$ and
 \begin{eqnarray*}
        \begin{cases}
          2b_1^2-b_1v_1-2v_1^2=0, \\
           2b_2^2-b_2v_2-2v_2^2=0.\nonumber\\
        \end{cases}
    \end{eqnarray*} 
We remark that $b\cdot v\neq 0$. We have then two cases:
\begin{enumerate}
\item $b\cdot v>0$: the canonical Jordan form of  $\ad_{e_{4}|_\n}$ up to scaling is
$\begin{pmatrix} 
        0 & 0 &0\\
        0 & 1&0\\
        0 & 0&-1\\
        \end{pmatrix},$ which corresponds to $\mathcal{A}_{3,4}\oplus \mathcal{A}_1.$
 
\item $b\cdot v<0$: the canonical Jordan form of  $\ad_{e_{4}|_\n}$ up to scaling is
$\begin{pmatrix} 
        0 & 0 &0\\
        0 & 0&1\\
        0 & -1&0\\
        \end{pmatrix},$ which corresponds to $\mathcal{A}_{3,6}\oplus \mathcal{A}_1$.
\end{enumerate}
Now, if we suppose that $v_1=0$ then $\theta=v_2\,e^2\neq 0$ implies that $v_2\neq 0.$ We can then deduce
that $b_1=A^1_1=A^2_2=A^1_2=A^2_1=0$. Because $J$ is non integrable, we have $b_2\neq 0.$ We conclude
that $2b_2^2-b_2v_2-2v_2^2=0,$ with $b_2v_2\neq 0.$ We obtain then the same canonical Jordan forms as above.

\end{proof}

\bibliographystyle{abbrv}


\end{document}